\documentclass[a4paper,reqno,11pt]{amsart}
\usepackage[T1]{fontenc}
\usepackage[utf8]{inputenc}
\usepackage[english]{babel}

\usepackage{amssymb, amsmath, amsthm, graphicx, color}
\usepackage[inline]{enumitem}
\usepackage[dvipsnames,svgnames,table]{xcolor}
\usepackage{comment}

\makeatletter
\def\thm@space@setup{
\thm@preskip=4mm
\thm@postskip=0mm
}
\makeatother

\usepackage{hyperref}
\hypersetup{
    pdftitle={Cops and robber in graphs with bounded vertex cover number},
    pdfauthor={},
    colorlinks,
    linkcolor={RoyalBlue},
    citecolor={RubineRed},
    urlcolor={blue!80!black}
}

\usepackage{cleveref}
\usepackage{mathtools}
\usepackage{thmtools, thm-restate}
\usepackage[longnamesfirst,numbers,sort&compress]{natbib}

\theoremstyle{plain}
\newtheorem{theorem}{Theorem}
\newtheorem{lemma}[theorem]{Lemma}

\newtheorem{observation}[theorem]{Observation}
\newtheorem{claim}[theorem]{Claim}
\newtheorem{problem}[theorem]{Problem}
\newcommand{\vomit}[1]{}
\newcommand{\tw}{\operatorname{tw}}

\newcommand{\cop}{\operatorname{cop}}

\newcommand{\Oh}{\mathcal{O}}

\DeclarePairedDelimiter\set{\{}{\}}

\newcommand{\E}{\mathbb{E}}
\renewcommand{\Pr}{\mathbb{P}}

\let\le\leqslant
\let\ge\geqslant
\let\leq\leqslant
\let\geq\geqslant

\let\subset\subseteq

\let\epsilon\varepsilon

\DeclareMathOperator\vc{vc}
\DeclareMathOperator\dist{dist}

\title[Cops and robber in graphs with bounded vertex cover number]{Cops and robber in graphs \\ with bounded vertex cover number}

\begin{document}

\author[P.~Bose]{Prosenjit Bose}
\address[P.~Bose]{School of Computer Science, Carleton University, Ottawa, Canada}
\email{jit@scs.carleton.ca}

\author[L.~Esperet]{Louis Esperet}
\address[L.~Esperet]{G-SCOP, CNRS, Université Grenoble Alpes, Grenoble, France}
\email{louis.esperet@grenoble-inp.fr}

\author[J.~Hodor]{Jędrzej Hodor}
\address[J.~Hodor]{Theoretical Computer Science Department, 
Faculty of Mathematics and Computer Science and  Doctoral School of Exact and Natural Sciences, Jagiellonian University, Krak\'ow, Poland}
\email{jedrzej.hodor@gmail.com}

\author[G.~Joret]{Gwena\"el Joret}
\address[G.~Joret]{D\'epartement d'Informatique, Universit\'e libre de Bruxelles, Belgium}
\email{gwenael.joret@ulb.be}

\author[P.~Micek]{Piotr Micek}
\address[P.~Micek]{Department of Theoretical Computer Science, Jagiellonian University, Kraków, Poland}
\email{piotr.micek@uj.edu.pl}

\author[C.~Rambaud]{Clément Rambaud}
\address[C.~Rambaud]{Universit\'e C\^ote d'Azur, CNRS, Inria, I3S, Sophia Antipolis, France}
\email{clement.rambaud@inria.fr}

\thanks{P.~Bose is supported in part by NSERC. L.\ Esperet is supported by the French ANR Projects
  TWIN-WIDTH
  (ANR-21-CE48-0014-01) and MIMETIQUE (ANR-25-CE48-4089-01), and by LabEx
  PERSYVAL-lab (ANR-11-LABX-0025). J.\ Hodor is supported by a Polish Ministry of Education and Science grant (Perły Nauki; PN/01/0265/2022). G.\ Joret is supported by the Belgian National Fund for Scientific Research (FNRS). P.\ Micek is supported by the National Science Center of Poland under grant UMO-2023/05/Y/ST6/00079 within the WEAVE-UNISONO program}

\begin{abstract}
Meyniel's conjecture states that $n$-vertex connected graphs have cop number $\Oh(\sqrt{n})$. The current best known upper bound is $n/2^{(1-o(1))\sqrt{\log n}}$, proved independently by Lu and Peng (2011), and by Scott and Sudakov (2011). In this paper, we extend their result by showing that every connected graph with vertex cover number $k$ has cop number at most $k/2^{(1-o(1))\sqrt{\log k}}$. This is the first sublinear upper bound on the cop number in terms of the vertex cover number. 
\end{abstract}

\maketitle

\section{Introduction}

In the \emph{cops and robber game} introduced by Quilliot \cite{Qui78}, Nowakowski and  Winkler \cite{NW83}, and later in its current form by Aigner and Fromme \cite{Aigner84}, cops are trying to catch a robber in a graph $G$. All cops and the robber are located on the vertices of $G$, and can move in turns to neighboring vertices. Initially (at step $0$), the cops choose their locations in $G$ (with possibly several cops on the same vertex), and then the robber chooses their location. At each subsequent step $i\ge 1$, each cop can either stay on the same vertex or move to an adjacent vertex, and after the cops' moves the robber can either stay on the same vertex or move to an adjacent vertex. The cops win the game if they can eventually catch the robber, i.e., at least one cop is located on the same vertex as the robber at some time during the game. 
Otherwise, the robber can avoid the cops forever and in this case the robber wins the game.
The \emph{cop number} of a graph $G$, denoted by $\cop(G)$, is the minimum number of cops required to win the game on $G$. 

The cops and robber game has been studied extensively, see e.g.\ the book by~\citet{BN11}. 
The most famous open problem in the area is Meyniel's conjecture from 1985, which states that $\cop(G)=\Oh(\sqrt{n})$ for every $n$-vertex connected graph $G$. See a survey on this conjecture by~\citet{survey}. The best known upper bound on the cop number of an $n$-vertex connected graph $G$ in terms of $n$ is $\cop(G)\le n/2^{(1-o(1))\sqrt{\log n}}$, obtained independently by \citet{Lu-Peng}, and by \citet{Scott-Sudakov} in 2011. Throughout this paper, $\log$ stands for the logarithm in base $2$, and the natural logarithm is
denoted by $\ln$.

Let us now turn our attention to classes of graphs closed under taking minors. 
\citet{Aigner84} showed with a beautifully simple argument that connected planar graphs have cop number at most $3$. A constant upper bound on the cop number holds in the more general setting of connected graphs excluding a fixed graph as a minor, as shown by \citet{And86}. See also Kenter, Meger, and Turcotte~\cite{KenterEtAl} for recent results in that direction. 

The cop number of connected graphs of genus $g$ is at most $(4g+10)/3$, as proved by \citet{genus}.
Similarly, the cop number of graphs of treewidth $k$ is at most $k/2 + 1$ as proved by \citet{Joret2008}.
However, it is not known whether these two bounds are asymptotically tight, the cop number could be sublinear in these parameters. In fact, for graphs of genus $g$, it was conjectured by Bonato and  Mohar~\cite{BM20} to be closer to $\Oh(\sqrt{g})$, and  
similarly, a $\Oh(\sqrt{k})$ bound might hold for graphs of treewidth $k$, as in Meyniel's conjecture.  
This motivates the following research problem: Is it possible to get a sublinear upper bound on the cop number of a connected graph $G$ in terms of a natural structural parameter of $G$ that is smaller than its number of vertices? In this paper, we make a first step in this direction.

A \emph{vertex cover} in a graph $G$ is a subset $S$ of vertices of $G$ such that every edge of $G$ is incident to a vertex of $S$.  
The \emph{vertex cover number} of $G$, denoted by $\vc(G)$, is the minimum size of a vertex cover in $G$. If one cop is located on each vertex of a vertex cover of a connected graph $G$, then when cops are moving for the first time, the robber is either already on a vertex with a cop, or is on a vertex adjacent to a vertex with a cop. This shows that $\cop(G)\le \vc(G)$ for every connected graph $G\ne K_1$. 
\citet{Gahlawat2023} recently improved this bound to $\cop(G) \leq \tfrac13\,\vc(G) + 1$.

Our main contribution is a sublinear upper bound on the cop number in terms of the vertex cover number, obtained by adapting the proof methods of Lu and Peng~\cite{Lu-Peng} and of Scott and Sudakov~\cite{Scott-Sudakov}. 

\begin{theorem}\label{thm:vc}
    For every connected graph $G\ne K_1$, we have
        \[\cop(G) \leq \frac{\vc(G)}{2^{(1-o(1))\sqrt{\log \vc(G)}}}.\]
\end{theorem}

Observe that for a graph $G$, a set $X \subset V(G)$ is a vertex cover if and only if every connected component of $G - X$ is a single vertex.
Jain, Kalyanasundaram, and Tammana~\cite{Jain24} studied the cop number when the above condition is relaxed, namely, when we allow each connected component to have at most two vertices.
They showed that in this case, $\cop(G) \leq |X| \slash 3 + 4$.
The technique that we use to prove~\Cref{thm:vc} allows us to obtain again a sublinear bound on the cop number as a function of $|X|$ in this setting as well. We indeed prove the following more general result, where we only need to assume that the diameter and cop number of all connected components of $G-X$ are bounded. 
In the remainder of the paper, whenever we talk about the \emph{diameter} of a subgraph $H$ of a graph $G$, we mean the maximum distance between two vertices of $H$ in $H$ (not  in $G$). 

\begin{theorem}\label{thm:technical}
    Let $G$ be a connected graph, let $d\ge 2$, $x\ge 2$ and $p\ge 1$ be integers, and let $X \subset V(G)$ with $|X| \leq x$ be such that for each connected component $C$ of $G - X$, the diameter of  $C$ is less than $d$ and the cop number of $C$ is at most $p$.
    Then,
        \[\cop(G)\le p+{x}\cdot{2^{-\sqrt{\log x-\Oh((\sqrt{\log x}+\log d)(\log \log x +d))}}}.\] 
        In particular, if for each connected component $C$ of $G-X$, the diameter of $C$ is $o(\sqrt{\log x})$ and the cop number of $C$ is at most $x\cdot 2^{-(1-o(1))\sqrt{\log x}}$, then we have $\cop(G)\le x\cdot 2^{-(1-o(1))\sqrt{\log x}}$.
\end{theorem}

\section{A key reduction}

Let $G$ be a connected graph.  
We say $t$ cops {\em can guard} a subset $X$ of vertices of $G$ if there is a constant $c$ and a strategy for these cops such that, irrespective of the initial positions of the cops, 
if the robber occupies a vertex in $X$ at some step $i$ of the game with $i\geq c$, then it is caught eventually by one of the cops. 

Observe that if a subset of the cops guards a subset $U$ of vertices in a connected graph $G$, then after a finite number of steps the robber is stuck forever in a fixed connected component of $G-U$. The remaining cops (not guarding $U$) can now easily reach this component and effectively start a new game in this component. In particular, the following observation is true.  

\begin{observation}\label{obs:removing_a_protected_set}
    Let $G$ be a connected graph,
    let $k,p$ be positive integers,
    and let $U,W \subseteq V(G)$. 
    Suppose that 
    \begin{enumerate}
        \item $k$ cops can guard $U$ in $G$, and
        \item for every connected component $C$ of $G-U$, the set $V(C) \cap W$ can be guarded by $p$ cops in $C$.
    \end{enumerate}
    Then $k+p$ cops can guard $W$ in $G$.
\end{observation}

\begin{proof}
We are going to describe a strategy for $k+p$ cops to guard $W$ in $G$. 
The cops may start from an arbitrary vertex. 
We split them into two sets $\mathcal{K}$ of size $k$ and $\mathcal{P}$ of size $p$. 
By assumption, there exists a time $t$ such
that if the robber occupies a vertex in $U$ after step $t$, they are caught.
Consequently, after step $t$, the robber is confined to a single connected
component $C$ of $G-U$.

Once the robber is confined to $C$, the cops in $\mathcal{P}$ move to $V(C) \cap W$ which by definition can be guarded by $p$ cops. If the robber visits a vertex
in $W \cap V(C)$, they are caught by $\mathcal{P}$. If the robber moves to a
vertex in $W \setminus V(C)$, they must enter $U$ and are caught by $\mathcal{K}$. Thus,
$W$ is guarded in $G$. 
\end{proof}

A \emph{geodesic} in a graph $G$ is a shortest path between two vertices of $G$. A classical result of Aigner and Fromme~\cite{Aigner84} is that a single cop can guard a geodesic.
\begin{lemma}[\citet{Aigner84}]\label{lemma:geodesic_protected}
    Let $G$ be a graph and let $P$ be a geodesic in $G$. 
    Then one cop can guard  $V(P)$ in $G$.    
\end{lemma}

We note that~\citet{Mohar24} characterized all isometric subgraphs that can be guarded by one cop, but for our purposes geodesics will be sufficient. 

The \emph{distance} between two vertices $u, v$ in a graph $G$, denoted by $\dist_G(u,v)$, is the minimum number of edges in a path between $u$ and $v$ in $G$. 
Given a vertex $v$ in $G$ and a nonnegative integer $d$, we define the ball of radius $d$ around a vertex and a set as follows:
\begin{align*}
B_G(v,d) &= \set{u\in V(G) \mid \dist_G(v,u)\leq d}, 
\intertext{and given a subset $A$ of vertices of $G$, we also define}
B_G(A,d) &= \bigcup_{v\in A} B_G(v,d). 
\end{align*}

The following lemma presents a strategy for cops which will work as a subroutine in the strategy witnessing~\Cref{thm:vc} (and~\Cref{thm:first_sublinear_bound}). 
Given a graph $G$, $X\subseteq V(G)$, and a positive integer $s$, if every vertex of $G$ contains at most $K$ vertices of $X$ in distance at most $s-1$, then 
for every vertex $v$ of $G$, the set $B_G(v,s)\cap X$ can be guarded by $2s\cdot K$ cops in $G$. 
We prove it by dividing a
set of cops into squads that rotate in shifts, we can ensure that a cop is
always dispatched to intercept the robber exactly when and where they attempt to
enter $X$ within the guarded ball.


\begin{lemma}\label{lemma:our_main_reduction}
    Let $s$ and $K$ be positive integers,
    let $G$ be a connected graph, and let $X \subseteq V(G)$.
    If for every $w \in V(G)$ we have that $|B_G(w,s-1) \cap X| \leq K$, 
    then, for every $v \in V(G)$, the set $B_G(v,s) \cap X$ can be guarded by $2s\cdot K$ cops in $G$.    
\end{lemma}

\begin{proof}
Assume that for every $w\in V(G)$ we have $|B_G(w,s-1)\cap X|\le K$. Let $v$ be a fixed vertex of $G$. We employ $2sK$ cops, divided into $2s$ patrols $P_0, \ldots, P_{2s-1}$, each consisting of $K$ cops.

We begin by moving all cops to vertex $v$. For the sake of numbering, we designate the step immediately following their arrival at $v$ as step 0. Thus, at the beginning of step 0, all cops are at $v$.

We define the operation of \emph{dispatching a patrol $P_j$ at step $i$ to a target set $A \subseteq B_G(v,s)$}, where $|A| \le K$, as follows. Fix a surjective function $\pi\colon P_j \to A$ (which exists since $|P_j|=K \ge |A|$). Each cop $c \in P_j$:
\begin{enumerate}
    \item Selects a geodesic $Q$ from $v$ to their assigned target $\pi(c)$. Let $d = \dist_G(v, \pi(c)) \le s$.
    \item Starts moving from $v$ along $Q$ at the beginning of step $i$, moving one edge per step towards $\pi(c)$. The cop reaches $\pi(c)$ after the cops' move at step $i+d-1$.
    \item Remains on vertex $\pi(c)$ until the end of step $i+s-1$.
    \item Starting at step $i+s$, moves back along $Q$ towards $v$, reaching $v$ by the end of step $i+s+d-1 \le i+2s-1$. The cop is then ready to be dispatched again at step $i+2s$.
\end{enumerate}

The cops' strategy is as follows: for each $i \ge 1$, let $r_i$ be the position of the robber at the beginning of step $i$. We dispatch the patrol $P_j$ with $j \equiv i \pmod{2s}$ to the target set $A_i = B_G(v,s) \cap B_G(r_i, s-1) \cap X$. Note that $|A_i| \le K$ by our assumption, so the dispatch is valid.

We claim that for any $i \ge s$, if the robber occupies a vertex $r_i \in B_G(v,s) \cap X$, they are caught immediately. Consider the state at step $i$. The robber was at position $r_{i-s+1}$ at step $i-s+1$. Since the robber moves at most one edge per step, their current position $r_i$ must be within distance $s-1$ of their position $s-1$ steps ago. Therefore:
\[ r_i \in B_G(r_{i-s+1}, s-1). \]
Combined with the assumption $r_i \in B_G(v,s) \cap X$, this implies that $r_i$ was included in the target set $A_{i-s+1}$ dispatched at step $i-s+1$. Specifically, a cop from patrol $P_{j'}$ (where $j' \equiv i-s+1 \pmod{2s}$) was assigned to the vertex $r_i$. Since $\dist_G(v, r_i) \le s$, this cop arrived at $r_i$ at or before step $(i-s+1) + s - 1 = i$. Thus, the cop is located at $r_i$ at step $i$, capturing the robber.
\end{proof}

\Cref{lemma:geodesic_protected} and \Cref{lemma:our_main_reduction} are enough to prove that the cop number of any graph is sublinear in its vertex cover number. We include the proof purely for exposition purposes as in the next section, we prove a stronger bound (\Cref{thm:vc} and \Cref{thm:technical}) with a significantly more intricate proof.

\begin{theorem}\label{thm:first_sublinear_bound}
    For every $\epsilon>0$, there is a constant $c$ such that
    for every connected graph $G$,
    \[
        \cop(G) \leq \epsilon \vc(G) + c.
    \]
\end{theorem}

\begin{proof} 
    Fix $\epsilon>0$ and let $c = (3/\epsilon)^{2/\epsilon}+1$, and  let $G$ be a connected graph.
    We proceed by induction on $\vc(G)$, with the base case $\vc(G)=0$ being trivial. 

    For the inductive case, assume that $\vc(G)\geq 1$ and that the theorem holds for smaller values of the vertex cover number. 
    Let $X$ be a vertex cover of $G$ of size $\vc(G)$.
    
    Suppose first that $G$ has diameter at least $\tfrac2{\epsilon}$. Then there is a geodesic $P$ in $G$ with $|V(P)|\ge \tfrac2\epsilon+1$. Note that every vertex cover intersects $V(P)$ in at least $\tfrac12(|V(P)|-1)$ vertices, and thus $|V(P) \cap X| \geq \frac{1}{\epsilon}$. 
    By \Cref{lemma:geodesic_protected}, $V(P)$ can be guarded in $G$ by a single cop.
    If $V(P) = V(G)$, then we are done, so suppose $V(G)\setminus V(P)\ne \emptyset$. Let $C$ be a connected component of $G-V(P)$ with maximum cop number. 
    By \Cref{obs:removing_a_protected_set},
    $\cop(G) \leq 1 + \cop(C)$.
    Note that $\vc(C) \leq |X| - |V(P)\cap X| \leq \vc(G) - \frac{1}{\epsilon}$.
    By the induction hypothesis,
    we deduce that
\[\cop(G)\le 1 + \cop(C) \le 1 + \epsilon\vc(C) + c\le 1 + \epsilon\left(\vc(G)-\frac{1}{\epsilon}\right) + c =\epsilon \vc(G) + c,\]
as desired. Therefore, in the remainder of the proof, we assume that $G$ has diameter less than $\frac{2}{\epsilon}$.

Now suppose that there is a vertex $w$ in $G$ with $|B_G(w,1) \cap X| \geq \frac{1}{\epsilon}$. 
Let $C$ be a connected component of $G - B_G(w,1)$ with maximum cop number. 
We have $\vc(C) \leq |X| - |B_G(w,1) \cap X| \leq \vc(G) - \frac{1}{\epsilon}$. Hence, since $B_G(w,1)$ can be guarded by one cop, by~\Cref{obs:removing_a_protected_set} and the induction hypothesis, we again obtain 
\[
\cop(G) \leq 1 + \cop(C) 
\leq 1 + \epsilon\vc(C) + c
\leq \epsilon\vc(G) + c.
\]
 
    Thus, we may assume that $|B_G(w,1) \cap X| < \frac{1}{\epsilon}$ for every vertex $w \in V(G)$. 
    Since the value on the left-hand side is an integer, we have 
    $|B_G(w,1) \cap X| \leq \lfloor\frac{1}{\epsilon}\rfloor$ for every vertex $w \in V(G)$. 
    In particular, applying \Cref{lemma:our_main_reduction} with $s=2$ and $K =  \lfloor\frac{1}{\epsilon}\rfloor$, we obtain that for every vertex $v\in V(G)$, 
    the set $B_G(v,2) \cap X$ can be guarded by $4K$ cops in $G$.

    Assume that there is a vertex $v \in V(G)$ with $|B_G(v,2) \cap X| \geq \frac{4}{\epsilon^2}$. By the induction hypothesis applied to $G - (B_G(v,2)\cap X)$ and \Cref{obs:removing_a_protected_set},
    we deduce that
\[        \cop(G)
        \leq 4K + \epsilon\left(\vc(G) - \frac{4}{\epsilon^2}\right) + c 
        = 4\left\lfloor\frac{1}{\epsilon}\right\rfloor + \epsilon\vc(G)-\frac{4}{\epsilon} + c \leq \epsilon \vc(G) + c.
\]
Therefore, we may assume that $|B_G(v,2) \cap X| < \frac{4}{\epsilon^2}$ for every vertex $v \in V(G)$. 

We now prove that 
\[
|B_G(v,i)\cap X|< \left(\frac{3}{\epsilon}\right)^i
\]
holds for every $v\in V(G)$ and $i\ge 1$. 
The proof goes by induction on $i$.  
Recall that we assumed $|B_G(v,1)\cap X| < \frac{1}{\epsilon}$ and $|B_G(v,2)\cap X| < \frac{4}{\epsilon^2}$.
Thus, the statement is clear for $i \in \{1,2\}$, and we suppose that $i \geq 3$. 
Consider some vertex $x\in B_G(v,i)\cap X$, let $y$ be a neighbor of $x$ that is in $B_G(v,i-1)$, and let $z$ be a neighbor of $y$ that is in $B_G(v,i-2)$. 
Observe that either $y\in X$, or if not, then $z\in X$, since the edge $yz$ is covered by $X$. 
We deduce that $x$ is either at distance at most 1 from the set $B_G(v,i-1)\cap X$, or at distance at most 2 from the set $B_G(v,i-2)\cap X$. It follows that  
\begin{align*}
    |B_G(v,i)\cap X|\le& \left|\bigcup \{ B_G(x,1) \cap X \mid x \in B_G(v,i-1) \cap X\}\right| \\
    + &\left|\bigcup \{ B_G(x,2) \cap X \mid x \in B_G(v,i-2) \cap X\}\right| \\ 
    <& \frac1\epsilon\cdot \left(\frac{3}{\epsilon}\right)^{i-1} +\frac4{\epsilon^2}\cdot \left(\frac{3}{\epsilon}\right)^{i-2}= \left(\frac13+\frac49\right)\left(\frac{3}{\epsilon}\right)^i \le \left(\frac{3}{\epsilon}\right)^i.
\end{align*}

Recall that we reduced to the case where $G$ has diameter less than $\frac{2}{\epsilon}$. Hence, for every vertex $v\in V(G)$, \[\vc(G)=|X| = |V(G) \cap X|=|B_G(v,\lfloor 2/\epsilon\rfloor)\cap X| < \left(\frac{3}{\epsilon}\right)^{2/\epsilon}.\]
As observed in the introduction, as $\vc(G)\ge 1$ we have $\cop(G)\le \vc(G)$, 
and thus $\cop(G) < \left(\frac{3}{\epsilon}\right)^{2/\epsilon} < c$.
    This concludes the proof. 
\end{proof}

The proof of \Cref{thm:first_sublinear_bound} actually shows that $\epsilon$ can be taken to be a 
 sufficiently large constant times $\tfrac{\log \log \vc(G)}{\log \vc(G)}$, and so we obtain that \[\cop(G)=\Oh\left(\vc(G)\cdot \tfrac{ \log \log \vc(G)}{\log \vc(G)}\right).\] 

\section{Proof of the main result}

In this section we prove the main result of this paper, Theorem~\ref{thm:technical}. 
Our proof strategy generalizes the framework established by Scott and Sudakov~\cite{Scott-Sudakov}, namely that if a graph has a large diameter, they reduce the graph by guarding a long geodesic with one cop and recurse on the remaining graph using \Cref{lemma:geodesic_protected}; if the diameter of the graph is small, they employ a probabilistic cop-placement strategy to capture the robber. We refine and generalize this approach. We begin by providing some additional intuition.

Let $G$ and $X$ be as in \Cref{thm:technical}.
The robber starts on an arbitrary vertex $v$. The ball $B_G(v,s)$ is the set of
vertices that the robber can potentially reach after $s$ steps. The main idea is
to analyze the growth rate of this ball restricted to a subset of vertices $X$
(i.e., the intersection $B_G(v,s) \cap X$). The set $X$ is essentially the
region that cops guard. We want to cover this intersection after a certain
number of steps to prevent the robber from safely entering $X$. Since $X$ acts
as a cut set (e.g., a vertex cover), this effectively confines the robber to a
small component of $G-X$, where they are easily caught by a few additional cops
(as formalized in Observation~\ref{obs:removing_a_protected_set}). For portions
of the restricted ball that grow quickly, we employ a deterministic patrolling
 strategy; for portions that grow slowly, we employ a probabilistic cop-placement strategy.

We formalize this intuition by refining the framework of Scott and Sudakov~\cite{Scott-Sudakov} based on the structure of $X$:

\begin{enumerate}

\item \textbf{Long Paths:} If $X$ contains many vertices of a long geodesic, we guard and remove the path using a single cop.



\item \textbf{A ball of fast growth in $X$:}
A major challenge in
adapting~\cite{Scott-Sudakov} is when $X$ does not contain many vertices of a long geodesic but
there is a ball whose intersection with $X$ is too big, which causes the probabilistic density
arguments to fail. We resolve this in \Cref{lemma:our_main_reduction}:
if $s$ is the smallest radius for which $B(v,s) \cap X$ is too big (relative to $s$)
for some $v \in V(G)$,
then \Cref{lemma:our_main_reduction} shows that $B(v,s) \cap X$ can be guarded 
by a deterministic patrolling strategy, and we can then apply induction to $G-(B(v,s)\cap X)$.

\item \textbf{Balls of slow growth in $X$:} If the intersection of all balls with $X$ is small (relative to the radius of the balls), we generalize the probabilistic
approach of Scott and Sudakov restricted to $X$ (\Cref{lem:main}) to show that
random cop placements capture the robber.

\end{enumerate}

Our main technical contribution is the following result.

\begin{lemma}\label{lem:main} 
    Let $k,d$ be integers with $k,d\ge 2$, and let 
    \[
    f(k,d):=2^{(\log k)^2 + 2^8 \log(kd)(\log \log k +d)}. 
    \]
    Let $G$ be a connected graph and let $X \subseteq V(G)$.
    Suppose that for each connected component $C$ of $G-X$, the diameter of $C$ is less than $d$. 
    Then  $X$ can be guarded by at most $|X|/k + f(k,d)$ 
    cops in $G$. 
\end{lemma}

\begin{proof}
We proceed by induction on $x=|X|$. If $x\le f(k,d)$ then the result clearly holds, as $X$ can be guarded by $x$ cops. So we may assume in the remainder of the proof that $x>f(k,d)$.

If there exists a geodesic $P$ in $G$ with $|V(P)\cap X|\ge k$, or a vertex $v\in V(G)$ with $|B_G(v,1)\cap X|\ge k$, then since such sets $V(P)$ or $B_G(v,1)$ can be guarded by a single cop in $G$, it follows from the induction hypothesis applied to $G - (V(P)\cap X)$ or $G - (B_G(v,1)\cap X)$, and Observation \ref{obs:removing_a_protected_set}, that $X$ can be guarded by at most 
\[
1+(x-k)/{k}+f(k,d)=x/k+f(k,d)
\]
cops in $G$, as desired. So we assume in the remainder of the proof that 
\begin{align}
\label{eq:geodesic_vs_X}
    |V(P)\cap X|&\le k-1&&\textrm{for every geodesic $P$ in $G$}
\intertext{and}
|B_G(v,1)\cap X|&\le k-1&&\textrm{for every $v \in V(G)$.}
\label{eq:1ball_vs_X}
\end{align}

Consider a geodesic $P$ in $G$, and observe that every subpath $P'$ of $P$ on $d+1$ vertices intersects $X$, since otherwise some connected component $C$ of $G-X$ would contain two vertices that have distance at least $d$ in $G$ and thus also in $C$, which is a contradiction with our initial assumption that $C$ has diameter less than $d$ in $G$. 
This implies 
$|V(P)\cap X|\ge \left\lfloor \tfrac{|V(P)|}{d+1}\right\rfloor \ge \tfrac{|V(P)|}{d+1}-1$, 
and, using~\eqref{eq:geodesic_vs_X}, 
\[
|V(P)|\le (d+1)|V(P)\cap X|+d+1 \le (d+1)k \le 2dk.
\]
Since this holds for every geodesic $P$ in $G$, we infer that 
\begin{equation}
\textrm{the diameter of $G$ is less than $2dk$.}
\label{eq:diameter-of-G}
\end{equation}

Assume that there exists $v \in V(G)$ and an integer $s \geq 1$ such that
    \[
    |B_G(v,s) \cap X| \geq s! \cdot 2^{s-1} k^{s}.
    \]
    Consider such a pair $(v,s)$ with $s$ minimum, and note that $s \geq 2$ by~\eqref{eq:1ball_vs_X}.
    By the minimality of $s$,
    we have $|B_G(w,s-1) \cap X| < (s-1)! \cdot 2^{s-2} k^{s-1}$ for every $w \in V(G)$.
    Therefore, by \Cref{lemma:our_main_reduction},
    the set $B_G(v,s) \cap X$ can be guarded by $2s \cdot (s-1)! \cdot 2^{s-2} k^{s-1} = s! \cdot 2^{s-1} k^{s-1}$ cops.
    Then, by \Cref{obs:removing_a_protected_set} and the induction hypothesis applied to $G - (B_G(v,s)\cap X)$,
    the set $X$ can be guarded by at most
    \[
        s! \cdot 2^{s-1} k^{s-1} + \frac{x - s!\cdot  2^{s-1} k^{s}}{k} + f(k,d) = \frac{x}{k} + f(k,d)
    \]
    cops in $G$, as desired. So we may assume in the remainder of the proof that
    \begin{equation*}
        |B_G(v,s) \cap X| < s! \cdot 2^{s-1} k^{s}<(2ks)^s \quad \quad \text{for all $v \in V(G)$ and $s \geq 1$,}
    \end{equation*}
    and in particular, using $d\ge 2$, we obtain
    \begin{equation}\label{eq:balls_are_small_2d+1}
        |B_G(v,3d+2) \cap X| \le |B_G(v,4d) \cap X|\le (8kd)^{4d} \quad \quad\text{for all $v \in V(G)$}.
    \end{equation}
    
    We set \[t:=\lceil\log (2dk)\rceil \ \quad \quad \text{ and} \ \quad \quad \alpha:=32 t \log(2t x).\]
We will need the following result, which is similar in spirit to \cite[Claim 2.3]{Scott-Sudakov} and follows from standard Chernoff bounds. We provide a detailed proof in \Cref{app} for completeness.
    \begin{claim}\label{cla:1} There exist $t$ 
    sets $C_1, \ldots,C_{t} \subseteq X$ such that for every $i\in [t]$,
    \begin{itemize}
        \item $|C_i|\le \tfrac{|X|}{kt}$, and
        \item for every $Y\subseteq X$ such that $|B_G(Y,2^{i}+d)\cap X|\ge \alpha k|Y|$, we have $|B_G(Y,2^{i}+d)\cap C_i|\ge |Y|$.
    \end{itemize}
    \end{claim}

    Next, we define pairs of subsets $A_i,D_i\subseteq X$ recursively for every $i\ge 1$ as follows. We start by setting $A_1:=B_G(r_0,2d+1)\cap X$, where $r_0$ is the initial location of the robber, and $D_1:=\emptyset$. For each $i\ge 2$, we define $A_{i}$ as an inclusion-wise maximal subset of $B_G(A_{i-1},2^{i-2}+d)\cap X$ such that  
    \begin{equation}\label{eq:1}
    |B_G(A_{i},2^{i-1}+d)\cap X|\le \alpha k |A_{i}|,
    \end{equation}
    
    and we set $D_{i}:=(B_G(A_{i-1},2^{i-2}+d)\cap X)\setminus A_{i}$. Note that the empty set satisfies condition~\eqref{eq:1}, so $A_i$ is well defined.

For every $i\in [t]$ and every vertex $v\in C_i$, we place a cop on $v$ at step 0 (vertices appearing in multiple sets $C_i$ receive multiple cops).
We claim that the cops can move in such a way that for every $i\in \{2, \dots, t+1\}$, 
every vertex of $D_i$ is occupied by a cop at the end of step $2^{i-1}+d$ (and this cop remains there until the end of the game). 
To see why this holds, observe that by the maximality of $A_{i}$, 
\[
|B_G(Y,2^{i-1}+d)\cap X|>\alpha k |Y| \quad \quad \textrm{ for every } i\in \{2, \dots, t+1\} \textrm{ and every } Y\subseteq D_{i} \textrm{ with } Y \neq \emptyset
\]
since otherwise $Y$ could be added to $A_{i}$, contradicting its maximality. 
It follows from~\Cref{cla:1} that  
\[
|B_G(Y,2^{i-1}+d)\cap C_{i-1}|\ge |Y| \quad \quad \textrm{ for every } i\in \{2, \dots, t+1\} \textrm{ and every } Y\subseteq D_{i}. 
\]
By Hall's marriage theorem \cite{Hall}, for every $i\in \{2, \dots, t+1\}$ there is an injective map $\phi_i\colon D_{i}\to C_{i-1}$  such that each vertex $v\in D_{i}$ lies at distance at most $2^{i-1}+d$ from $\phi_i(v)$ in $G$. This suggests the following strategy for the cops: For every $i\in \{2, \dots, t+1\}$ and every vertex $v\in D_i$, a cop located on $\phi_i(v)\in C_{i-1}$ moves towards $v$ along a geodesic between $\phi_i(v)$ and $v$ as soon as the game starts, reaching $v$ at step (at most) $2^{i-1}+d$, and stays there until the end of the game. Note that we only need $t\cdot \tfrac{|X|}{kt}\le \tfrac{|X|}{k}$ cops for this strategy, and this guarantees that for every $i\in \{2, \dots, t+1\}$, at the end of step $2^{i-1}+d$, every vertex of $D_i$ is occupied by a cop, who stays there until the end of the game. 
We remark that this is trivially true for $i=1$ as well, since $D_1=\emptyset$. 

For every $i\ge 1$, set 
\[
N_i:=A_i\cup \bigcup_{j\in [i]} D_j. 
\] 
As $A_i\subseteq B_G(A_i,2^{i-1}+d)\cap X = A_{i+1} \cup D_{i+1}$, 
we have $N_i\subseteq N_{i+1}$ for every $i\ge 1$.
We now prove that, as a consequence of the strategy above, for every $i\in [t+1]$, at the end of step $2^{i-1}+d$, 
\begin{enumerate}[label=(\Alph*)]
    \item\label{eq1} if the robber is in $X$, then it is in $A_i$, 
    \item\label{eq2} otherwise, the robber is in a connected component $C$ of $G-X$ such that $N_G(V(C))\subseteq N_i$.
\end{enumerate}
We prove the result above by induction on $i\ge 1$.
We start with the base case $i=1$. Recall that $A_1=B_G(r_0,2d+1)\cap X$, where $r_0$ denotes the starting location of the robber. If the robber lies in $X$ at the end of step $2^{1-1}+d=d+1$, then as \[B_G(r_0,d+1)\cap X\subseteq B_G(r_0,2d+1)\cap X=A_1,\] the robber lies in $A_1$. Assume now that the robber lies in some connected component $C$ of $G-X$ at the end of step $d+1$, say on some vertex $r_1$. As $C$ has diameter less than $d$ in $G$, all the vertices of $C$ lie in $B_G(r_1,d-1)\subseteq B_G(r_0,2d)$, and thus $N_G(V(C))\subseteq B_G(r_0,2d+1)\cap X=A_1=N_1$  (recall that $D_1=\emptyset$).

We now consider the case $i\ge 2$. By the induction hypothesis, at the end of step $2^{i-2}+d$, the robber either lies in $A_{i-1}$, or in some connected component $C$ of 
$G-X$ such that $N_G(V(C)) \subseteq N_{i-1}$. If the robber lies in a connected component $C$ of $G- X$ at the end of step $2^{i-2}+d$ and is also in $C$ at the end of step $2^{i-1}+d$, then we have $N_G(V(C))\subseteq N_{i-1}\subseteq N_i$ and the desired property holds. So we may assume that the robber lies in $A_{i-1}$ at the end of step $2^{i-2}+d$, or exits a connected component $C$ of $G-X$ at some step $s$ with $2^{i-2}+d+1\le s \le 2^{i-1}+d$, meaning that the robber was in $C$ at the end of step $s-1$ but is no longer in $C$ at the end of step $s$. 
Recall that at the end of step $2^{i-2}+d$, every vertex of $\bigcup_{j\in [i-1]}D_{j}$ is occupied by a cop and remains occupied until the end of the game. This implies that after step $2^{i-2}+d$, the robber can only exit $C$ through a vertex of $N_{i-1}\setminus \bigcup_{j\in [i-1]}D_{j}\subseteq A_{i-1}$. 
It follows that the robber is in $A_{i-1}$ at the end of some step $s$ with $2^{i-2}+d\le s \le 2^{i-1}+d$. As a consequence, the robber lies in $B_G(A_{i-1},2^{i-2})\subseteq B_G(A_{i-1},2^{i-2}+d)$ at the end of step $2^{i-2}+2^{i-2}+d=2^{i-1}+d$. It follows from the cops' strategy that at the same moment, every vertex of $D_i$ is occupied by a cop, which implies that the robber must indeed be located in $B_G(A_{i-1},2^{i-2}+d)\setminus D_i$. If the robber is in $X$, then it must be in $(B_G(A_{i-1},2^{i-2}+d)\cap X)\setminus D_i=A_i$, as desired. It only remains to consider the case where the robber is in a connected component $C'$ of $G-X$ at the end of step $2^{i-1}+d$, say on some vertex $r$. As the diameter of $C'$ is less than $d$, every vertex of $C'$ is at distance less than $d$ from $r$, and in particular $N_G(C')\subseteq B_G(r,d)\cap X$. Since $r\in B_G(A_{i-1},2^{i-2})$, this shows that $N_G(C')\subseteq B_G(A_{i-1},2^{i-2}+d)\cap X=A_i\cup D_i\subseteq N_i$, which concludes the proof that \ref{eq1} and \ref{eq2} hold at the end of step $2^{i-1}+d$.

The final goal is to show that $A_{t+1}=\emptyset$. By \ref{eq1} and \ref{eq2}, this implies that at the end of step $2^{t}+d$, the robber lies in some connected component $C$ of $G-X$ such that all the vertices of $N_G(V(C))$ are occupied by a cop, and thus the robber remains stuck in $C$ forever. This shows that $X$ is indeed guarded by at most $|X|/k$ cops and concludes the proof. 

To show that $A_{t+1}$ is empty, we argue by contradiction and assume that it is not. 
Recall that $A_1=B_G(r_0,2d+1)\cap X$ and \[A_2\subseteq B_G(A_1,d+1)\cap X\subseteq B_G(r_0,3d+2)\cap X.\]
By \eqref{eq:balls_are_small_2d+1}, this implies that $|A_2|\le (8kd)^{4d}$. Moreover, for every $i\ge 2$, \eqref{eq:1} implies that
\[|A_{i+1}|\le |B_G(A_i,2^{i-1}+d)\cap X|\le \alpha k |A_i|\] 
This implies that $|A_{t+1}|\le (\alpha k)^{t-1} |A_2|\le (\alpha k)^{t-1}\cdot (8kd)^{4d}$,
and thus \[|B_G(A_{t+1},2^{t}+d)\cap X|\le \alpha k |A_{t+1}|\le (\alpha k)^{t}\cdot (8kd)^{4d}.\]

On the other hand, recall that $t=\lceil \log (2dk)\rceil$ and that $G$ has diameter at most $2dk \le 2^t\le 2^t+d$, and hence since $A_{t+1}$ is non-empty, $|B_G(A_{t+1},2^{t}+d)\cap X|=|X|=x$. We obtain \[x\le (\alpha k)^{t}\cdot (8kd)^{4d}.\]
Taking the logarithm on both sides, we have
\begin{eqnarray}\label{eq:logx1}
    \log x \le  t \log (\alpha k)+4d \log (8kd). 
\end{eqnarray}
Recall that  $t=\lceil \log (2dk)\rceil\le \log (2dk)+1=\log (4dk)$ and
\begin{eqnarray*}
    \alpha=32 t \log(2t x)& \le & 32 \log(4dk) (\log x + \log (2\log (4dk)))\\
    & \le & 32 \log x \cdot (\log x + \log \log x +1)\\
    & \le & 64 \log^2 x,
\end{eqnarray*}
where we have used $\log x\ge \log (f(k,d))\ge \log(4dk)$. As a consequence, \[\log( \alpha k)\le \log k+\log (64 \log^2x)=\log k+2 \log \log x +6\le \log k+4 \log \log x,\] since $x\ge f(k,d)\ge 2^{8 }$.
Plugging this into \eqref{eq:logx1}, we obtain 
\begin{eqnarray}\label{eq:logx2}
\log x \le   \log(4dk)(\log k+4 \log \log x)+4d \log (8kd).
\end{eqnarray}
Writing $\log (4dk)=2+\log d +\log k$ and $\log (8dk)=3+\log d +\log k$ and expanding the two products in the right hand side of \eqref{eq:logx2}, we obtain a sum  of 8 terms, 
each of which is at most $16d \log d \cdot \log^2k \cdot \log \log x \le 16d^2 \log^2k \cdot\log \log x$.  It follows that 
\[\log x \le 2^7 d^2 \log^2 k \cdot \log \log x.\]
Taking the logarithm on both sides again, we obtain 
\begin{eqnarray*}
\log \log x & \le & 7+ 2\log \log k+2 \log d + \log \log \log x\\
& \le & 7+ 2\log \log k+2 \log d + \tfrac12 \log \log x,
\end{eqnarray*}
and thus $\log \log x  \le  14+ 4\log \log k+4 \log d$.
Substituting this in \eqref{eq:logx2}, we obtain
\begin{eqnarray*}
\log x & \le &  \log(4dk) (\log k+56+ 16\log \log k+16 \log d)+4d \log (8kd)\\
& < & \log^2k + 2^8 \log(kd)(\log \log k +d).
\end{eqnarray*}
This contradicts the assumption that $x=|X|\ge f(k,d)=2^{\log^2k + 2^8 \log(kd)(\log \log k +d)}$ and concludes the proof.
\end{proof}

We now deduce \Cref{thm:technical} from Lemma \ref{lem:main}.


\noindent {\it Proof of \Cref{thm:technical}.}
We set \[k=2^{\sqrt{\log x -2^9 (\sqrt{\log x}+ \log d)(\log \log x +d)}}.\] By Lemma \ref{lem:main}, $X$ can be guarded by at most $x/k+f(k,d)$ cops. As each connected component of $G-X$ has cop number at most $p$, it follows that $\cop(G)\le x/k+f(k,d)+p$.

\smallskip

As $k\le 2^{\sqrt{\log x}}$, we have $\log k\le \sqrt{\log x}$ and $\log \log k\le \log \log x$. It follows that
\begin{eqnarray*}
f(k,d) & = & 2^{\log x -2^9 (\sqrt{\log x}+ \log d)(\log \log x +d)}\cdot 2^{2^8 (\log k+\log d)(\log \log k +d)}\\
& \le  & 2^{\log x -2^9 (\sqrt{\log x}+ \log d)(\log \log x +d)}\cdot 2^{2^8 (\sqrt{\log x}+\log d)(\log \log x +d)}\\
& \le  & 2^{\log x -2^8 (\sqrt{\log x}+\log d)(\log \log x +d)}\\
& \le  & x\cdot 2^{-2^8 \sqrt{\log x}}\\
& \le & x/k.
\end{eqnarray*}
It follows that $\cop(G)\le 2x/k+p=x\cdot 2^{-\sqrt{\log x-\Oh((\sqrt{\log x}+\log d)(\log \log x +d))}}+p$, as desired.
\hfill $\Box$


Theorem \ref{thm:vc} follows directly from \Cref{thm:technical} by setting $p=1$ and $d=2$ (if $X$ is a vertex cover, then each connected component $C$ of $G-X$ is a single vertex, so $\cop(C)=1$ and $C$ has diameter $0<2$). 

\section{Conclusion}

A natural question is whether a similar approach can be applied to the graph parameters that have been mentioned in the introduction, such as genus or treewidth. An interesting intermediate step between the vertex cover number and the treewidth of a graph is its treedepth, which is defined as follows. Given a rooted tree $T$, its \emph{closure} is the graph obtained from $T$ by adding an edge between each vertex and all its descendants in the tree. The \emph{vertex-height} of a rooted tree is the maximum number of vertices on a root-to-leaf path. The \emph{treedepth} of a connected graph $G$, denoted by $\mathrm{td}(G)$, is the minimum vertex-height of a rooted tree whose closure contains $G$ as a subgraph. 
Note that $\tw(G)+1\le \mathrm{td}(G)\le \vc(G)+1$ for every connected graph $G$.

\begin{problem}
    Is it true that for connected graphs $G$, $\cop(G)=o(\mathrm{td}(G))$?
\end{problem}

At the moment we do not quite see how to improve significantly on $\cop(G)\le \tfrac12\, (\mathrm{td}(G)+1)$, which can be easily deduced from results on the cop number of graphs of bounded treewidth. 

\subsection*{Acknowledgments}

This research was initiated at the Twelfth Annual Workshop on Geometry and Graphs held at the Bellairs Research Institute in February 2025. Thanks to the organizers and the other workshop participants for creating a productive working atmosphere.

\bibliographystyle{abbrvnat}
\bibliography{bib}

\appendix
\section{Proof of Claim \ref{cla:1}}\label{app}

We start by recalling the following classical  concentration inequalities, see e.g.\ Section A.1 in \cite{AS16}.

\begin{lemma}[Chernoff bounds]\label{lem:chernoff}
    Let $X_1, \dots, X_n$ be mutually independent random variables taking values in $\{0,1\}$,
    and let $\mu = \E\left[\sum_{i=1}^n X_i\right]$.
    \begin{enumerate}
        \item $\Pr\left[\sum_{i=1}^n X_i \geq 2\mu\right] \leq \exp\left(-\mu/3\right)$.
        \item $\Pr\left[\sum_{i=1}^n X_i \leq \mu/2\right] \leq \exp\left(-\mu/8\right)$.
    \end{enumerate}
\end{lemma}

We are now ready to prove Claim \ref{cla:1}, which we restate here for convenience. We take the notation and assumptions from the proof of Lemma \ref{lem:main}, in particular \[x=|X|\ge f(k,d),\quad t=\lceil\log (2dk)\rceil, \quad \text{ and } \alpha=32 t \log(2t x).\]

\begin{claim} There exist $t$ sets $C_1, \ldots,C_{t}$ such that for every $i\in [t]$,
    \begin{itemize}
        \item $|C_i|\le \tfrac{|X|}{kt}$, and
        \item for every $Y\subseteq X$ such that $|B_G(Y,2^{i}+d)\cap X|\ge \alpha k|Y|$, we have $|B_G(Y,2^{i}+d)\cap C_i|\ge |Y|$.
    \end{itemize}
\end{claim}

\begin{proof} Let $C_1, \ldots,C_{t}$ be random subsets of $X$ defined as follows: For each $i\in [t]$ and $v\in X$, we place $v$ in $C_i$ independently with probability $\tfrac1{2kt}$. 

We start by showing that the first item fails with probability at most $\tfrac14$. Note that for each $i\in [t]$, $\E\left[|C_i|\right]=\tfrac{x}{2kt}$, so it follows from Lemma \ref{lem:chernoff} that $\Pr\left[|C_i| \geq \tfrac{x}{kt}\right] \leq \exp\left(-x/6kt\right)$. Hence, the probability that at least one of the sets $C_i$ is larger than $x/kt$ is at most \[t \cdot \exp\left(-x/6kt\right)\le \tfrac14,\] where the inequality follows from the assumption that $x\ge f(k,d)\ge 6kt \log(4t)\ge 6kt \ln(4t)$.

We now show that the second item fails with probability at most $\tfrac12$, which implies that with probability at least $1-\tfrac12-\tfrac14\ge \tfrac14$ both items hold and thus there exist sets $C_1, \ldots,C_{t}$ satisfying the desired properties. Consider a set $Y\subseteq X$. Assume that $Z=B_G(Y,2^i+d)\cap X$ has size at least $\alpha k|Y|$. Then for any $i\in [t]$,  $\E\left[|Z\cap C_i|\right]=\tfrac{|Z|}{2kt}\ge \tfrac{\alpha}{2t}\,|Y|$. Since $\alpha>4t$, it follows from  Lemma \ref{lem:chernoff} that \[\Pr\left[|Z\cap C_i| <|Y|\right]\le \Pr\left[|Z\cap C_i| \leq \frac{\alpha}{4t}\,|Y|\right] \leq \exp\left(-\frac{\alpha}{16t}\,|Y|\right).\] 
Hence, the probability that there exist $i \in [t]$ and a subset $Y\subseteq X$ such that $|B_G(Y,2^{i}+d)\cap X|\ge \alpha k|Y|$ and $|B_G(Y,2^{i}+d)\cap C_i|< |Y|$ is at most 
\begin{eqnarray*}
    t\cdot \sum_{s=1}^x \binom{x}{s}\exp\left(-\frac{\alpha s}{16t}\right) & \le & t\cdot \sum_{s=1}^x \binom{x}{s}\exp\left(-\frac{\alpha s}{16t}\right)\\
    & \le & t\cdot \sum_{s=1}^x \exp \left(s \ln x-\frac{\alpha s}{16t} \right)\\
    & \le & t\cdot \sum_{s=1}^x \exp \left(-\frac{\alpha s}{32t} \right) \hspace{0.3cm} \text{(since $\alpha \ge 32 t \log x\ge 32 t \ln x$)}
    \\
    & \le & tx \cdot \exp \left(-\frac{\alpha }{32t} \right) \\& \le & 1/2 \hspace{2.4cm} \text{(since $\alpha/32t = \log(2t x)\ge \ln(2tx)$)}. 
\end{eqnarray*}
This concludes the proof of the claim.
\end{proof}
\end{document}